\documentclass{amsart}

\usepackage{amsthm,amsmath,amsfonts,amssymb}
\usepackage{xargs} % for multiple optional arguments
\usepackage{mathtools}
\usepackage{enumitem}

\usepackage[left=1.4in, right=1.4in, bottom=1in]{geometry}

\numberwithin{equation}{section}
\theoremstyle{plain}                % title and number in bold, text italic
\newtheorem{theorem}{Theorem}[section]

\newtheorem{proposition}[theorem]{Proposition}

\theoremstyle{remark}
\newtheorem{assumption}[theorem]{Assumption}

\newtheorem{remark}[theorem]{Remark}

\providecommand{\alias}{}
\renewcommand{\alias}[1]{\providecommand{#1}{}\renewcommand{#1}}

\DeclarePairedDelimiter\ab{\langle}{\rangle} % <angle brackets>
\DeclarePairedDelimiter\abs{\lvert}{\rvert}   % |absolute value|
\DeclarePairedDelimiter\norm{\lVert}{\rVert}  % ||norm||
\DeclarePairedDelimiter\bkt{[}{]}             % [brackets]
\DeclarePairedDelimiter\brc{\{}{\}}           % {braces}
\DeclarePairedDelimiter\prn{(}{)}             % {parentheses}

\providecommand\giv{}

\newcommand\given\giv % alias

\DeclarePairedDelimiterX\set[1]\{\}{ #1 }
\DeclarePairedDelimiterX\sets[2]\{\}{ #1\,:\,#2 }
\DeclarePairedDelimiterXPP\pp[1]{\mathbb{P}}[]{}{
   \renewcommand\giv{\nonscript\:\delimsize\vert\nonscript\:\mathopen{}}
   #1}
\DeclarePairedDelimiterXPP\ppup[2]{\mathbb{P}^{#1}}[]{}{
   \renewcommand\giv{\nonscript\:\delimsize\vert\nonscript\:\mathopen{}}
   #2}
\DeclarePairedDelimiterXPP\ppdown[2]{\mathbb{P}_{#1}}[]{}{
   \renewcommand\giv{\nonscript\:\delimsize\vert\nonscript\:\mathopen{}}
   #2}

\DeclarePairedDelimiterXPP\ee[1]{\mathbb{E}}[]{}{
   \renewcommand\giv{\nonscript\:\delimsize\vert\nonscript\:\mathopen{}}
   #1}
\DeclarePairedDelimiterXPP\eeup[2]{\mathbb{E}^{#1}}[]{}{
   \renewcommand\giv{\nonscript\:\delimsize\vert\nonscript\:\mathopen{}}
   #2}
\DeclarePairedDelimiterXPP\eedown[2]{\mathbb{E}_{#1}}[]{}{
   \renewcommand\giv{\nonscript\:\delimsize\vert\nonscript\:\mathopen{}}
   #2}

\let\preexp\exp
\let\exp\relax
\DeclarePairedDelimiterXPP\exp[1]{\preexp}(){}{#1}

\newcommandx{\prf}[2][2=t\in\bkt{0,T}]{ \set{#1}_{#2}}
\newcommandx{\seq}[2][2=n\in\N]{\set{#1}_{#2}}

\newcommandx{\fml}[2]{\set{#1}_{#2}}
\newcommandx{\tol}[1]{\stackrel{#1}{\longrightarrow}}
\newcommandx{\eqd}[1][1=d]{\stackrel{(#1)}{=}}
\newcommandx{\neqd}[1][1=d]{\stackrel{(#1)}{\neq}}

\makeatletter
\let\oldabs\abs \def\abs{\@ifstar{\oldabs}{\oldabs*}}
\let\oldab\ab \def\ab{\@ifstar{\oldab}{\oldab*}}
\let\oldnorm\norm \def\norm{\@ifstar{\oldnorm}{\oldnorm*}}
\let\oldbkt\bkt \def\bkt{\@ifstar{\oldbkt}{\oldbkt*}}
\let\oldbrc\brc \def\brc{\@ifstar{\oldbrc}{\oldbrc*}}
\let\oldprn\prn \def\prn{\@ifstar{\oldprn}{\oldprn*}}
\let\oldpp\pp \def\pp{\@ifstar{\oldpp}{\oldpp*}}
\let\oldppup\ppup \def\ppup{\@ifstar{\oldppup}{\oldppup*}}
\let\oldppdown\ppdown \def\ppdown{\@ifstar{\oldppdown}{\oldppdown*}}
\let\oldee\ee \def\ee{\@ifstar{\oldee}{\oldee*}}
\let\oldeeup\eeup \def\eeup{\@ifstar{\oldeeup}{\oldeeup*}}
\let\oldeedown\eedown \def\eedown{\@ifstar{\oldeedown}{\oldeedown*}}

\let\oldexp\exp \def\exp{\@ifstar{\oldexp}{\oldexp*}}
\makeatother

\alias{\R}{{\mathbb R}}
\alias{\C}{{\mathbb C}}
\alias{\Z}{{\mathbb Z}}
\alias{\N}{{\mathbb N}}
\alias{\Nz}{{\mathbb N}_0}

\DeclareMathOperator\Id{Id}

\newcommand{\RN}[2]{\frac{d #1}{d #2}}
\newcommand{\tRN}[2]{\tfrac{d #1}{d #2}}

\newcommand{\oo}[1]{\frac{1}{#1}}

\newcommand{\too}[1]{\tfrac{1}{#1}}
\newcommand{\tot}{\too{2}}

\newcommand{\inds}[1]{ 1_{\set{#1}}} % indicator of a curly set

\newcommand{\upto}{\nearrow}

\newcommand{\tint}{\textstyle\int}

\newcommand{\eps}{\varepsilon}
\newcommand{\ld}{\lambda}

\newcommand{\sB}{\mathcal{B}}

\newcommand{\sF}{\mathcal{F}} \newcommand{\bF}{\mathbb{F}}

 \newcommand{\bL}{\mathbb{L}}
\newcommand{\sM}{\mathcal{M}}

\newcommand{\sP}{\mathcal{P}} \newcommand{\bP}{\mathbb{P}}
 \newcommand{\bQ}{\mathbb{Q}}

\newcommand{\lone}{\bL^1}
\newcommand{\ltwo}{\bL^2}
\newcommand{\lpee}{\bL^p}

\newcommand{\lque}{\bL^q}

\newcommandx{\pds}[2][1= ]{\frac{\partial #1}{\partial #2}}
\newcommandx{\tpds}[2][1= ]{\tfrac{\partial #1}{\partial #2}}

\newcommand{\eas}{\text{ a.s.}}

\newcommand{\efor}{\text{ for }}
\newcommand{\eforall}{\text{ for all }}
\newcommand{\eforeach}{\text{ for each }}
\newcommand{\eand}{\text{ and }}

\newcommand{\ewhere}{\text{ where }}
\newcommand{\ewith}{\text{ with }}

\newcommand{\itos}{It\^ o's}

\newcommand{\tbe}{\tilde{\beta}}
\newcommand{\tA}{\tilde{A}}
\newcommand{\sMl}{\sM^{loc}}

\newcommand{\MQ}{M^{\bQ}}
\newcommand{\MPx}{M^{\bP,\xi}}
\newcommand{\MP}{M^{\bP}}
\newcommand{\loz}{\bL^{1,0}}
\newcommand{\ltz}{\bL^{2,0}}
\newcommand{\lto}{\bL^{2,1}}

\newcommand{\lpq}{\bL^{p,q}}
\newcommand{\lpo}{\bL^{p,1}}
\newcommand{\hbeta}{\hat{\beta}}
\newcommand{\hbe}{\hat{\beta}}

\newcommand{\eeq}[1]{\eeup{\bQ}{#1}}
\newcommand{\eeqst}[1]{\eeup*{\bQ}{#1}}
\newcommand{\eep}[1]{\eeup{\bP}{#1}}

\newcommand{\un}{u^{(n)}}
\DeclareMathOperator{\Prog}{Prog}

\begin{document}
\title{Representation of random variables as   Lebesgue integrals}

\author{Sara Biagini}
\thanks{The authors would like to thank Fausto Gozzi for
    suggesting some related literature. They are also grateful to
    the anonymous referees, and the Editors for a number of
constructive comments that improved the quality of this paper.}
\address{Sara Biagini, Department of Economics and Finance,
LUISS University, Rome, Italy}

\author{Gordan \v{Z}itkovi\'{c}}
\address{Gordan \v{Z}itkovi\'{c}, Department of Mathematics, University of Texas at Austin,
Texas, US }

\maketitle

\begin{abstract}
We study representations of a random variable $\xi$
as an integral of an adapted process with respect to the Lebesgue measure.
The existence of such representations in two different regularity classes
is characterized in terms of  the quadratic variation of  (local)
martingales closed by $\xi$.
\end{abstract}

% \smallskip
% \noindent \textbf{Keywords.} Lebesgue-integral representation,
% martingale representation, local martingales, Girsanov's theorem, 
% quadratic variation.

%% START
\section{Introduction}
This paper focuses on the representation of a random variable as an
adapted Lebesgue - as opposed to stochastic -  integral. We start the analysis with
statements of our main results and then place them in the extant
literature while  offering motivation for their study.\\
\indent Let $(\Omega, (\mathcal{F}_t)_{t\leq T}, \sF,\bP)$ be a filtered
probability space. Given an $\sF_T$-mea\-su\-ra\-ble random variable $\xi$, we
ask whether there exists a progressively-measurable process $\beta$
such
that
\begin{align}\label{equ:ac-repr}
    \xi = \int_0^T \beta_u\, du, \eas
\end{align}
with $\beta$ in a given integrability class. We focus on the Lebesgue
measure on a finite time-horizon $[0,T]$ because other settings
(alternative measures instead of the Lebesgue measure, alternative
horizons, or the discrete time on an infinite horizon instead of the
continuous time) lead to a similar analysis.
\\
\indent Our main results apply to two integrability classes for $\beta$, but we
discuss interesting features of some other classes, too, in Section
\ref{sec:examples}.
We say that $\beta$ is
\emph{weakly regular} if $$\int_0^T \beta_u^2\, du < \infty \ \   a.s. ,$$  and
\emph{strongly regular} if $$\ee{ \int_0^T \beta_u^2\, du}<\infty.$$
Assuming throughout that all  {$\sF$}-local martingales are continuous, we show in Theorem
\ref{thm:main-strong} that the representation \eqref{equ:ac-repr} holds for
some strongly regular $\beta$ if and only if $\xi \in \lone$ and
\begin{align*}
   \ee{\int_0^T \oo{T-t}\, d\ab{M}_t}<\infty
   \ewhere M_t = \ee{\xi \giv \sF_t}.
\end{align*}
In a less restrictive, weakly regular case, our Theorem \ref{thm:main-weak}
states that \eqref{equ:ac-repr} holds for a weakly regular $\beta$ if
and only if there exists a probability measure $\bQ$ equivalent to $\bP$
and a $\bQ$ local martingale $M$ with $M_T=\xi$ such that
\begin{align}
   \label{equ:slow}
   \int_0^T \oo{T-t}\, d\ab{M}_t<\infty,\eas
\end{align}
Intuitively, an absolutely continuous representation of the form
\eqref{equ:ac-repr} with a weakly regular $\beta$ exists if and only if
$\xi$ closes a local martingale whose quadratic variation grows slowly
enough at $T$. This problem has a interesting link with  the so-called ``fundamental theorem of asset pricing''
(see Theorem 1.1, p.~487 of \cite{DelSch94}). As is well known in the Mathematical Finance community, this Theorem
states that a locally-bounded semimartingale $M$ is a
local martingale under some measure $\bQ$ equivalent to $\bP$ if and
only if it satisfies the condition  {of No Free Lunch with Vanishing Risk (NFLVR in the sequel)}. NFLVR is a slightly stronger version of
the classical NA (No Arbitrage) condition of Mathematical Finance. We may think, informally,
 of a process that
satisfies NFLVR as a  {measure-free version of a local martingale,}   or,
similarly, as a semimartingale whose local-martingale part is everywhere
 more active  than its finite-variation part. \\
\indent  When focusing on the representation \eqref{equ:ac-repr} of $\xi$ under the
weaker, probability-free, condition on $\beta$, that question boils down to the
relationship between $\xi$, the set of null events, and the filtration. Rephrased in financial terms, what we show is that  \eqref{equ:ac-repr} holds if and only if $\xi$ closes a price process which has the  property and moreover is a  ``slow"
local martingale under a suitable $\bQ$ - in the sense of \eqref{equ:slow}. Such ``slow" local martingale
that converges to $\xi$ can be used as a proxy for the
good approximability of $\xi$ by $\sF_t$-adapted random variables as
$t\upto T$.\\ 
\indent 
Unlike in the case of martingale representation, the question of
uniqueness of an absolutely continuous representation admits a trivially
negative answer in many interesting integrability classes, including
both weak and strong regularity discussed above. That fact served as a
prompt to look for a canonical, rather than unique $\beta$. When
$\ee*{\int_0^T \beta_u^2\, du}<\infty$ is required, the $\beta$ that
minimizes   $\ee*{\int_0^T \beta_u^2\, du}$ admits an easy-to-verify
explicit form, namely
\[
      \hbe_t = \frac{1}{T} M_0 + \int_0^t \oo{T-u}\, dM_u, \  t\in [0,T),
      \]
where $M_t = \ee{\xi \giv \sF_t}$.  Unfortunately, we could not
identify an analogous natural notion of canonicity in the weakly regular
case.

% Before illustrating and comparing our contribution with the  literature,
% we  give  extra motivation for  the analysis and the scope of our work.
Absolutely continuous representation issues  arise quite easily in
applications. For instance, in \cite{AidBia23} the authors deal with a
linear-quadratic stochastic control problem on the Wiener space, arising
from carbon regulation. In that problem, the controls are square
integrable \emph{rates}, i.e., state dynamics involve integrals of these
controls with respect to $dt$.
Furthermore, the objective function contains a
terminal penalty term which is a function of an integral of one of
the controls, $\beta$, so that the random variable  $\xi = \int_0^T
\beta_t\, dt$ appears in the objective function.  Since the problem is not
strictly convex in $\beta$, the authors of \cite{AidBia23} were
 only able  {to} obtain an explicit
expression for the optimal $\hat \xi$, and for the associated martingale
$\hat M_t = \ee*{ \hat \xi \giv \sF_t}$. They left the problem of finding an  optimal,
square integrable, rate
$\hat \beta$ that \emph{represents} the optimal $\hat \xi$
open (see \cite{AidBia23}, Remark 4.1).

Integrable-enough absolutely continuous representations come in handy in
other contexts, as well. For example,  they   provide  useful estimates
when proving existence of solutions to stochastic differential
equations. The interested reader can consult Chapter 6 of the \cite{FGS17}  
for a general treatment, or \cite{BGZ22}  for an application 
to stochastic delayed differential equations in an optimal investment problem.\\ 
\indent
The only existing result concerning absolutely-continuous representation
we are aware of is the ``factorization formula''
of Da Prato and Zabczyk (see Theorem 5.2.5, p.~58 in \cite{DaPZab96}).
Set on an abstract Wiener space, it provides an explicit
absolutely continuous representation of a random variable given by a
stochastic integral. It relies on a version of a stochastic Fubini
theorem (see Theorem  {4.18} of \cite{DapZab14}) but does not address the
regularity of the representation itself, or provide any necessary conditions.
A deeper discussion of why their approach, based on the stochastic
Fubini theorem, does not lead to the kinds of results we are interested in is
given in Remark \ref{rem:Fubini}.

Our results extend the existing ones in several directions. First, we
give \emph{necessary and sufficient} conditions on the random variable
$\xi$ for the representation to exist under both weak and strong
regularity. Furthermore, in the strongly regular case we show that the unique
\emph{martingale} solution  { of } the representation problem arises as the
$\ltwo$-norm minimizer on the product space.

The paper is organized as follows: Section 2 treats the strongly regular
and Section 3 the weakly regular case; Section 4 contains further
examples, results and comments.\\ 
\indent \textbf{Setup and notation.} We consider a measurable space $(\Omega,\sF)$,
together with a maximal family $\sP$ of mutually equivalent probability
measures on $\sF$, as well as a right-continuous filtration
$\bF=\prf{\sF_t}$,  { with $\sF_0$ $\sP$-trivial.}
 {When we write that a filtration is \emph{generated by a
    Brownian motion  {W}}, we
always have the usual right-continuous
and complete augmentation of the natural filtration in mind. On the
other hand, a filtration $\prf{\sF_t}$ is said to \emph{support} a Brownian
motion if there exists an
$\prf{\sF_t}$-Brownian motion $W$.}

{We say that a  process $M$ is a
\emph{$\sP$-local martingale}
if it is a local martingale under some $\bP\in\sP$, and we denote
the set of all $\sP$-local martingales by $\sMl$.
We impose the
following, standing, assumption throughout:
\begin{assumption}
   \label{asm:cont}
   Each $\sP$-local martingale is continuous.
\end{assumption}
 {In particular, the above assumption implies  that for each $M\in \sMl$ there exists a unique process $\langle M \rangle $ such that $M^2-\langle M \rangle $ is also in $ \sMl$.}
\begin{remark}
\label{rem:palet}
According to Theorem 5.38, p.~155 in \cite{HeWanYan92}, continuity
of all martingales on a filtered probability space is equivalent to
the requirement that all
$\prf{\sF_t}$-stopping times be predictable. Since this property
stays invariant under equivalent measure changes, we conclude that
Assumption \ref{asm:cont} holds if we only ask that there
\emph{exists} a single probability measure $\bP \in \sP$ such that all
$\bP$-local martingales are continuous. 
\end{remark}
 }

For $\bP \in \sP$, $\lpee(\bP)$ is a shorthand for $\lpee(\Omega,\sF,
\bP)$ while $\lpq(\bP)$,  $q \in [0,\infty)$  denotes the set of all
$\prf{\sF_t}$-predictable processes $\beta$ with $  \int_0^T
\abs{\beta_u}^p\, du \in \lque(\bP)$.
When $p\geq 1$, the space $\bL^{p,1}(\bP)$ comes
with the norm: $$\norm{\beta}_{\lpo(\bP)}=\mathbb{E}^{\bP}\left[ \ \int_0^T
\abs{\beta_u}^p\, du\right ] ^{1/p},$$ while no topology on $\bL^{p,0}(\bP)$
will be needed. Since the spaces $\bL^{0}(\bP), \bL^{p,0}(\bP)$, $\bP \in
\sP$ coincide, we omit the probability measure from the notation and
simply write $\bL^{0}, \bL^{p,0}$.

For $\xi \in \sF_T$ and $\bP \in \sP$, we set
\begin{align*}
  \sB^{p,q}(\xi,\bP) &:= \Big\{ \beta \in \lpq(\bP)\,:\,
     \tint_0^T \beta_u\, du = \xi\eas \Big\}.
\end{align*}
When $q=0$, we omit the measure $\bP$ and write only
$\sB^{p,0}(\xi)$.

\section{The strongly regular case}
In this section we choose and fix a probability measure $\bP \in \sP$
and use it as the underlying measure in all probabilistic
statements. In particular, we write:
$\ltwo$ and $\bL^{2,1}$
for $\ltwo(\bP)$ and $\bL^{2,1}(\bP)$ respectively;  {and $ \sB^{2,1}(\xi)$ in place of $\sB^{2,1}(\xi,\bP)$.}
\begin{theorem}
   \label{thm:main-strong}
   For $\xi\in\lone$,
   let $\{M_t\}_{t\in [0,T]}$
   and $\{\hbe_t\}_{t\in [0,T)}$ be defined by
   \begin{align}
      \label{equ:mb-m}
      M _t &= \ee{ \xi \giv \sF_t},\  t\in [0,T] \\
      \label{equ:mb-b}
      \hbe_t &= \frac{1}{T} M_0 + \int_0^t \oo{T-u}\, dM_u, \  t\in [0,T).
   \end{align}
  The following statements are equivalent under Assumption \ref{asm:cont}:
   \begin{enumerate}[label*=\arabic*.]
      \item $\sB^{2,1}(\xi) \ne \emptyset$.
      \item $\hbe \in \sB^{2,1}(\xi)$.
      \item $\hbe \in \lto$.
      \item $\ee*{\int_0^T \oo{T-t}\, d\ab{M}_t}<\infty$.
   \end{enumerate}
 When $\sB^{2,1}(\xi) \ne \emptyset$, the process $\hbe$ given by
\eqref{equ:mb-b} is, up to a version,
\begin{enumerate}
\item[(a)] the unique martingale on $[0,T)$ in
$\sB^{2,1}(\xi)$
   \item[(b)] the minimal $\bL^{2,1}$-norm element in $\sB^{2,1}(\xi)$.
\end{enumerate}
\end{theorem}
\begin{proof}
   \emph{1.$\to$ 2.}
   Assuming that $\sB^{2,1}(\xi)$ is nonempty
   consider the minimization problem
   \begin{align}
      \label{equ:inf}
      \inf_{\beta\in \sB^{2,1}(\xi)} \ee{\int_0^T \beta_u^2\, du}
      = \inf_{\beta\in \sB^{2,1}(\xi)} \norm{\beta}_{\mathbb{L}^{2,1}}^2.
   \end{align}

The set $\sB^{2,1}(\xi)$ is convex and closed in $\lto$.  By
intersecting it with a large-enough ball in $\lto$, we may assume that
it is also bounded in $\lto$.  The Banach Alaoglu theorem ensures then
that such restricted subset  of $\sB^{2,1}(\xi)$ is   \emph{weakly
compact}. Since the $\lto$ norm is a weakly lower semicontinuous
function, so is its square and thus there exists a  $\tbe$ which attains
the minimum in \eqref{equ:inf}. This minimizer is also unique by strict
convexity of the objective function.
%We pick a minimizing sequence
%$\seq{\betan}$ and extract a weakly-convergent subsequence (still
%denoted by $\seq{\betan}$) with the weak limit $\tbe \in \lto$.
%For $Z \in \ltwo$, we set $\zeta_t =\ee{ Z \giv
%\sF_t}$, so that $\zeta \in \lto$. \my{An application of the integration by parts formula then gives:}
%   \begin{align*} \ee{ Z \xi} & = \ee{ Z \int_0^T \betan_u\, du}
%      = \ee{ \int_0^T \zeta_u \betan_u\, du} \to
%      \ee{ \int_0^T \zeta_u \tbe_u\, du}  \\
%      & = \ee{ Z \int_0^T \tbe_u\, du}.
%   \end{align*}
%This is true for sufficiently many $Z$ to
%conclude that $\int_0^T \tbe_u\, du = \xi$, a.s., i.e., that $\tbe \in
%\sB^{2,1}(\xi)$.  Moreover,
\\
 \indent {The rest of the proof of this implication is organized as follows.
    We start by showing
that  the minimality of $\tbe$   implies that $\tbe$ is orthogonal
in $\bL^{2,1}$ to a sufficiently rich class of processes. Using this
result, we prove that $\tbe$ is a martingale. Finally,  we apply
It\^{o}'s formula  and find  that (a modification of) $\tbe$
coincides with $\hbe$.}
\\
\indent
In order to show martingality of $\tbe$, we perturb it in the
direction of a process $\gamma \in \bL^{2,1}$ with $\textstyle\int_0^T
\gamma_u\, du = 0$, a.s.
 {By construction, the processes
    $\tbe\pm \eps \gamma$ belong to $\sB^{2,1}(\xi)$, which implies that:
    $$\norm*{\tbe\pm
\eps \gamma}^2_{ {\bL^{2,1}}}\geq
\norm*{\tbe}^2_{ {\bL^{2,1}}},$$ for each $\eps\in\R_+$. Writing
down the relevant expectations,  the inequality becomes
$$\norm*{\tbe}^2_{ {\bL^{2,1}}}  \pm\, 2 \eps \ee{\int_0^T
    \tbe_u\gamma_u \, du} +\eps^2 \ee{\int_0^T \gamma_u^2\, du} \geq
    \norm*{\tbe}^2_{ {\bL^{2,1}}}. $$
Simplifying and sending
$\eps$ to zero,} we get the following set of ``first-order
conditions''
\begin{align}
    \label{equ:variational}
    \ee{ \int_0^T \tbe_u \gamma_u\, du} = 0, \ \forall\,
    \gamma\in \bL^{2,1} \ewith \int_0^T \gamma_u\, du = 0, \text{ a.s.}
\end{align}
   Given $t < s$ in $[0,T)$, for each  $\sF_t$-measurable
   and random variable $\chi \in \ltwo$ we define
   \begin{align*}
      \gamma^{\chi}_u = \begin{cases}
                            {  0 } &  {u\in [0,t] } \\
                           \chi \too{s-t},  & u \in (t,s], \\
                           -\chi \too{T-s}, & u \in (s,T].
                        \end{cases}
   \end{align*}
   so that $\int_0^T \gamma^{\chi}_u\, du =0$ for each $\chi$.
   By applying the equality
   in \eqref{equ:variational} to $\gamma^{\chi}$ for all $\sF_t$-measurable $\chi
      \in \ltwo$ we obtain
   \begin{align*}
   \ee*[\Big]{ \frac{\tA_s - \tA_t}{s-t} \giv \sF_t} =
\ee*[\Big]{ \frac{\tA_T -\tA_s}{T-s} \giv \sF_t}\, \text{ a.s.}\, ,
   \end{align*}
   where $\tA_t = \int_0^t \tbe_u\, du$.
    {Since $\tbe \in
       \sB^{2,1}(\xi)$, $ \tA_T = \xi = M_T$, with $M$ given by \eqref{equ:mb-m}.
   Slightly rearranging,  we obtain:
   \begin{align}
      \label{equ:iden}
  \oo{s-t}\Big(\ee*[\big]{ \tA_s \giv \sF_t} -\tA_t  \Big) =
\oo{T-s} \Big(M_t - \ee*[\big]{ \tA_s \giv \sF_t}  \Big) \eas
   \end{align}
   Since the right-hand side of \eqref{equ:iden} is
   a martingale in $t$ on $[0,s)$, so is the left-hand side. In
   particular, the finite-variation part in its semimartingale
   decomposition, given via integration by parts by
  \begin{align*}
  \int_0^t \left( \frac{1}{(s-u)^2}\big( \ee*[\big]{ \tA_s \giv \sF_u}
-\tA_u\big)\, - \frac{1}{s-u}  \tbe_u \right)\, du,
  \end{align*}
  must vanish for all $t<s$, a.s. Consequently,
$$\tbe_t = \frac{ \ee*[\big]{ \tA_s \giv \sF_t} - \tA_t}{s-t}\text{
a.s., for almost all $t<s<T$}.$$
Passing to the limit $s\uparrow T$ on the right hand side above, we obtain}

   \begin{align}
      \label{equ:betadrift}
  \tbe_t = \frac{ \ee*[\big]{ \tA_T \giv \sF_t} - \tA_t}{T-t} =  \frac{  M_t - \tA_t}{T-t} =\frac{  M_t - \int_0^t \tbe_u\,du}{T-t}.
   \end{align}
   It follows that $\tbe$ has a continuous version on $[0,T)$, which we, from now
   on, adopt. Furthermore, the right-hand side of \eqref{equ:betadrift} is a
   semimartingale on $[0,T)$
   so we can use \itos{} formula once more to conclude that
   \[ \tbe_t = \oo{T} M_0 + \int_0^t \oo{T-u}\, dM_u \efor t\in [0,T).\]
   Therefore, $\hbe = \tbe$ and statement \emph{2.} follows immediately.
   \medskip

   \emph{2.}~$\to$ \emph{3.}~ Immediate.
   \medskip
    {
    {\emph{3.}~$\to$~\emph{4.}}
   With $\tau_n = \inf\sets{t\geq 0}{\ab{M}_t \geq n}$, Fubini's
theorem implies that}
 {
   \begin{equation}
       \label{Fub-chain}
      \begin{split}
              \ee{\int_0^{T} (\hbe_{\tau_n \wedge  u} - \hbe_0)^2\, du}
              & = \int_0^T \ee{\int_0^{\tau_n \wedge u}
                  \oo{(T-t)^2}\,
          d\ab{M}_t}\, du\\
              & = \ee{\int_0^T\int_t^T
                  \inds{t\leq \tau_n} \oo{(T-t)^2}\,
                  du\,
              d\ab{M}_t}\\
              &= \ee{\int_0^{\tau_n} d\ab{M}_t \int_t^{T}   \oo{(T-t)^2}\,du
              } =  \ee{ \int_0^{\tau_n} \oo{T-t}\, d\ab{M}_t}.
      \end{split}
   \end{equation}
}
    {
   Since $\hbe^{\tau_n}$ is an $\ltwo$-bounded martingale
   on $[0,t]$ for each $t<T$, $(\hbe^{\tau_n}-\hbe_0)^2$ is a submartingale on
   the same domain, and the optional sampling
   theorem implies that
   \begin{align*}
       \ee{ (\hbe_{\tau_n \wedge t} - \hbe_0)^2} \leq \ee{
       (\hbe_t - \hbe_0)^2} \eforeach t<T.
   \end{align*}
   Thus, by \eqref{Fub-chain},
   \[ \ee{ \int_0^{\tau_n} \oo{T-t}\, d\ab{M}_t} =
\int_0^T \ee{ (\hbe_{\tau_n \wedge u} - \hbe_0)^2}\, du \leq
\int_0^T \ee{ (\hbe_u - \hbe_0)^2}\, du = \norm*{\hbe-\hbe_0}_{\lto} ^2 <
\infty,\]
and it suffices to let $n\to\infty$ and use the monotone convergence
theorem. }

\medskip

 {   { \emph{4.}~$\to$~\emph{3.}}
   We let $n\to\infty$ in \eqref{Fub-chain} and use
   Fatou's lemma on the left-hand side and the monotone
   convergence theorem on the right to conclude that
  \begin{align*}
  \norm*{\hbe}_{\lto} \leq \abs*{\hbe_0} +
  \norm*{\hbe - \hbe_0}_{\lto} \leq \abs*{\hbe_0}+
  \ee{ \int_0^{T} \oo{T-t}\, d\ab{M}_t}^{1/2}<\infty.
\end{align*}
}

\smallskip
   \emph{3.}~$\to$~\emph{1.} It suffices to show that
   $\int_0^T \hbe_t\, dt = \xi$, a.s. The definition of $\hbe$ in
   \eqref{equ:mb-b}
   and integration by parts imply that for $t\in [0,T)$ we have
   \begin{align}
      \label{equ:ip-1}
      (T-t) \hbe_t = T \hbe_0  + \int_0^t \,dM_u - \int_0^t \hbe_u\, du
      = M_t - \int_0^t \hbe_u\, du.
   \end{align}
  Another round of integration by parts, but this time applied to the stochastic integral
   $\int_0^t \oo{T-u}\, dM_u$, implies that
   \begin{align}
      \label{equ:ip-2}
      \hbe_t = \oo{T-t} M_t + \int_0^t \frac{M_u}{ (T-u)^2}\, du.
   \end{align}
   Put together, identities \eqref{equ:ip-1} and \eqref{equ:ip-2}, give
   \begin{align*}
      \int_0^t \hbe_u \, du =
      \frac{ \int_0^t  \frac{M_u}{(T-u)^2}\, du,}{ \oo{T-t} } \efor
      t\in [0,T),
   \end{align*}
   and the final step is to use l'H\^ opital's rule and the fact that $M_t\to \xi$, as $t\to
      T$.
   %\textcolor{red}{Therefore, we are done if we show that, for $t\rightarrow T$ the left-hand side goes to $0$ a.s. By contradiction, suppose that the left-hand side does not converge to $0$ a.s. and  fix  a set $A\in \mathcal{F}_T, \bP(A)>0$ where this happens. We show that this contradicts $\hbe\in \lto$. For a.s. $\omega \in A$ then,   there exists a positive constant $k(\omega) $    such that the   Borel set
%    $ C(\omega) :=\{ t \in [0,T] \mid (T-t)|\hbe_t(\omega)| >k(\omega )\}$  verifies:
%     $$ \mathcal{L}\left( C(\omega)  \cap \left[T-\frac{1}{n}, T\right] \right) >0 , $$
%     where $\mathcal{L}$ denotes the Lebesgue measure.
%   We can assume without loss of generality that $k$ is bounded,  $0<k(\omega)\leq c$. Then,
%     $$E\left[\int_0^T \hbe^2_t dt \right] \geq E\left [ I_A \int_{C } \frac{k^2}{(T-t)^2} dt\right ] = \infty,$$
% which is   the desired contradiction with the assumption $\hbe\in \lto$.
%   }\\
\\ \vspace{0.1cm} \indent
 Concerning the last part of the statement of the theorem,  (b) was established
 in the course of the proof of \emph{1.}~$\to$~\emph{2.} above. For (a), we assume that
there exists another martingale $\beta^*$ in
$\sB^{2,1}(\xi)$ so that
$$ M_t = \ee{\int_0^T \beta^*_u\,du \giv \sF_t} = \int_0^t \beta^*_u\, du
+ \beta^*_t (T-t).$$
The equality \eqref{equ:ip-1} above implies that
$$ 0= \int_0^t (\beta^*_u - \hat{\beta}_u)\, du
+ (\beta^*_t - \hat{\beta}_t) (T-t) \eforall t\in [0,T), \eas $$
It follows that $\beta^*_t - \hbe_t$ is continuously differentiable for
$t\in [0,T)$, and the conclusion $\beta^* = \hbe$ follows  by differentiation.
\end{proof}
 {
\begin{remark}
\label{rem:Volterra}
As the anonymous referee observes,
equation \eqref{equ:betadrift}
can be interpreted as a (pathwise) Volterra-type equation of the
second kind:
\begin{align*}
    u(t) = f(t) + \int_0^{T} K(t,s) u(s)\ ds, \ewhere f(t) =
    \frac{M_t}{T-t} \eand K(t,s) = -\oo{T-t} \inds{s\leq t}.
\end{align*}
A \emph{formal} iterative solution
(obtained by repeatedly replacing $u(\cdot)$ on the
{right-hand} side by the whole right-hand side and then taking the
limit) can be
written as $\tilde\beta_t = \lim_n \un(t)$, where
\begin{align*}
    \un(t) = f(t) + \sum_{i=0}^{ {n}} \int K^i(t,s) f(s)\, ds
\end{align*}
and $K^i$ is the $i$-th composition power of $K$, i.e.,
$K^{0} = K$ and
\begin{align}
    \label{Ki}
    K^i(t,s) &= \int_0^T  \dots \int_0^T K(t,s_1) K(s_1,s_2)\dots
    K(s_i,s)\, ds_1\, \dots \, ds_i \efor i\geq 1.
\end{align}
Substituting the explicit expression $K(t,s) = -(T-t)^{-1}
\inds{s\leq t}$ into \eqref{Ki} above, we obtain
\[ K^i(t,s)= (-1)^{i}(T-t)^{-1} \int_0^T  \dots \int_0^T \oo{T-s_1} \dots
             \oo{T-s_i} \inds{ t \geq s_1 \geq \dots \geq s_i \geq s
         } \, ds_1\, \dots \, ds_i. \]
The last iterated integral is taken over the simplex
$\Delta = \{ (s_1, \dots,  s_{i}) \in
[s,t]^{i-1}\, : \, s_1\leq s_2 \dots \leq s_{i}\}$, and
the function $\prod_{j=1}^i (T-s_k)^{-1}$ inside the integral is
symmetric in $s_1,\dots,
s_{i-1}$, so, for $s\leq t$ we have
\begin{align*}
    K^i(t,s)
    &= (-1)^{i} \oo{T-t}\, \oo{i!} \int_s^t\dots \int_s^t
    (T-s_1)^{-1} \dots (T-s_i)^{-1}\, ds_1\, \dots\, ds_i\\ &
    = (-1)^{i} \oo{T-t} \oo{i!} \Big(\int_s^t (T-u)^{-1}\,
     du\Big)^i=
      \oo{T-t} \oo{i!} \Big(\log \frac{T-t}{T-s}\Big)^i,
\end{align*}
and consequently,
\[ \sum_{i=0}^{\infty} K^i(t,s) =  \oo{T-s} \inds{s\leq t}.\]
This implies that the formal solution $\tilde{\beta}_t$ takes the form
\begin{align}
    \label{mmm}
    \tilde{\beta}_t = \frac{M_t}{T-t}  +  \int_0^t
    \frac{M_s}{(T-s)^2}\, ds,
\end{align}
which, after integrating by parts in the last integral, matches
\eqref{equ:mb-b}.
\end{remark}
}

\section{The weakly regular case}

 {Fix an $\mathcal{F}_T$ measurable random variable  $\xi \in \bL^0$},  and  let $\sMl(\xi)$ denote the set of all $M\in \sMl$
such that $M_T = \xi$.  Let $\sP^1(\xi)$ be the set of probabilities in $\mathcal{P}$ which integrate $\xi$.\footnote{ {The set $\sP^1(\xi)$ is not empty, and it can be proved as in the proof of the next Theorem, arrow
\emph{1.}~$\to$~\emph{2.}, by setting $\beta=0$.}}
For $\bP \in\sP^1(\xi)$, we set
\[ \MPx_t = \eep{\xi \giv \sF_t},\ t\in [0,T],\]
taken in its continuous version, so that $\MPx$ is the unique
$\bP$-martingale in $\sMl(\xi)$.  Finally, for $M \in \sMl$
we define
\begin{align}
   \label{equ:beM}
   \hbeta^M_t = \frac{1}{T} M_0  + \int_0^t \oo{T-u}\, dM_u, \  t\in [0,T).
\end{align}

\begin{theorem}
   \label{thm:main-weak}
   For an  $\mathcal{F}_T$ measurable random variable  $\xi \in \bL^0$, the following are equivalent:
   \begin{enumerate}[label*=\arabic*., itemsep = 0.5em]
      \item $\sB^{2,0}(\xi) \ne \emptyset$.
     \item  {$\sB^{2,1}( \xi,\mathbb{Q})\ne \emptyset$} for some $\mathbb{Q}\in\sP^1(\xi)$.
      \item $\hbeta^M \in \bL^{2,0}$ for some $M \in \sMl(\xi)$.
      \item $\int_0^T \oo{T-t}\, d\ab{M}_t<\infty \,\eas$,
      for some $M \in \sMl(\xi)$.
\item $\eeqst{ \int_0^T \oo{T-t}\, d\ab{M^{\mathbb{Q},\xi}}_t }<\infty $,  for some $\mathbb{Q}
\in \sP^1(\xi)$.
   \end{enumerate}
\end{theorem}
\begin{proof}

   \emph{1.}~$\to$~\emph{2.}
   We pick $\beta \in \sB^{2,0}(\xi)$ and  {$\bP\in \sP$, and define
   $\mathbb{Q} \in \sP$} by
   \begin{align*}
      \frac{d\mathbb{Q}}{d\mathbb{P}}=  c \frac{1}{1+\abs{\xi}+\int_0^T \beta_u^2\, du},
   \end{align*}
   where $c$ is the normalizing constant. This way $\mathbb{Q}\in \sP^1(\xi)$  and the process $\beta$ belongs to
   $\lto(\mathbb{Q})$, and, hence, also to $\sB^{2,1}(\mathbb{Q})$.
   
\medskip
   
\emph{2.}~$\to$~\emph{5.} This is the content of the implication
   \emph{1.}~$\to$~\emph{4.}~in Theorem \ref{thm:main-strong}, but,
   possibly, under an equivalent probability measure.
\medskip

   \emph{5.}~$\to$~\emph{4.} Immediate.

\medskip

   \emph{4.}~$\to$~\emph{3.} Let
   $M\in \sMl(\xi)$ be as in the statement, and let
   $\mathbb{Q}\in\sP$ be such that $M$ is a $\mathbb{Q}$-local martingale.
   We define the nondecreasing sequence
   $\{T_m\}_{m\in\N}$ of stopping times by
\[ T_m = \inf\sets*{t\geq 0}{ \int_0^t \oo{T-u}\, d\ab{M}_u \geq m},\]
so that $\mathbb{Q}[ T_m < T] \to 0$, as $m\to\infty$. The process $\hbeta^M$,
given by \eqref{equ:beM} above, is a continuous local $\mathbb{Q}$-martingale
on $[0,T)$ so there exists another nondecreasing sequence $\seq{\tau_n}$ of
stopping times with the property that $\tau_n \to T$, a.s., such that
$(\hbeta^M)^{\tau_n}$ is an $\ltwo(\mathbb{Q})$-bounded martingale on $[0,T]$. In
particular,
   \[ \eeq{ (\hbeta^M_{u\wedge T_m \wedge \tau_n})^2} = \eeq{
          {\langle \hbeta^{M}\rangle_{u\wedge T_m \wedge \tau_n}}}
         \text{ for all $m,n \in \N$ and $u\in
         [0,T)$.} \]
   We let $n\to\infty$ and use Fatou's lemma together with the monotone convergence theorem   to  conclude that
   \begin{align}
      \label{equ:fatou}
      \eeq{ (\hbeta^M_{T_m \wedge u})^2 } \leq \eeq{\ab*{\hbeta^M}_{T_m \wedge u}}
      = \eeq{
         \int_0^{ T_m \wedge u} \oo{(T-r)^2}\, d\ab*{M}_r}
   \end{align}
   for $u<T$ and $m\in\N$.
    {By
   using the inequality
  $$\int_0^{T_m} (\hbeta^M_u)^2\, du \leq \int_0^T
(\hbeta^M_{T_m\wedge u})^2\, du \ \  a.s.,$$
   integrating \eqref{equ:fatou} above in $u$ over $[0,T]$,
 and applying Fubini's Theorem under the product measure $du \otimes
 d\ab{M}_r$ we obtain
   \begin{align*}
       \eeq{ \int_0^{T_m} \big(\hbeta^M_{ {u}}\big)^2\, du} &\leq
     \eeq{ \int_0^T \big(\hbeta^M_{T_m \wedge u }\big)^2\, du}
     = \eeq{ \int_0^T \int_0^T \inds{r \leq u \wedge T_m} \oo{(T-r)^2}\,
     d\ab{M}_r\, du
     } \\
     &= \eeq{ \int_0^{T_m}
\oo{(T-r)^2}
     d\ab{M}_r
\int_r^{T}  \, du\,
     }
         = \eeq{ \int_0^{T_m}  \oo{T-r}\, d\ab{M}_r}\leq m.
   \end{align*}
}
Since $\mathbb{Q}[T_m =
T] \to 1$ as $m\to\infty$, we conclude that $\int_0^T (\hbeta^M_u)^2\, du <
\infty$, a.s.

\medskip

\emph{3.}~$\to$~\emph{1.} The last argument in the proof of Theorem
\ref{thm:main-strong} is based only on the integration by parts
formula and on the property  {$\hbeta \in  \loz$. Therefore  it can be
applied here, since   $\hbeta^M \in \ltz \subseteq \loz$.}
   % shows that
%   \begin{align*}
%      \lim_{t\to T} \int_0^t \hbeta^M_u\, du = \xi,
%   \end{align*}
%for any $\mathbb{Q}\in \sP^1(\xi)$. Since $\hbeta^M\in \ltz
%\subseteq \loz$, we have
%   $\int_0^T \hbeta^M_u\, du = \xi$, a.s.
\end{proof}

\begin{remark}\
\label{rem:Fubini}
As it aims for generality, but also operates
within specific regularity classes, our proof of Theorem
\ref{thm:main-weak} above  does not use the \emph{stochastic Fubini
theorem},  {a joint name for a class of statements about the
    permissibility of the interchange of a Lebesgue and a stochastic integral under different
sets of conditions.  We refer the reader to
Theorem  {4.18} in \cite{DapZab14} or
Theorem 2.2 of \cite{Ver12} for two versions referred to later in this paper.}

 {To provide a more detailed explanation,} let us start with a brief
description of how an argument based on it would play out. Under
Assumption \ref{asm:cont}, it would start with a choice of a measure
$\bQ\in\sP$ as in    Theorem \ref{thm:main-weak}, item \emph{2.}~and the associated  martingale
$M^{\bQ,\xi}$, where  we assume,  without loss of generality,  that
$M_0=\eeq{\xi}=0$. When its conditions are satisfied, the stochastic Fubini
theorem,
applied to the function
$\psi(s,t, \omega) = \frac{1}{T-t}I_{[0,s]}(t)$ and with integrals with
respect to $dM$ and $ds$, yields
\begin{align*}
\int_0^T ds \int_0^s \frac{1}{T-t}\, dM_t = \int_0^T dM_t
\,\frac{1}{T-t} \int_t^T ds = M_T = \xi,
\end{align*}
making
\begin{align}
   \label{equ:cand-beta}
\beta_s = \int_0^s \frac{1}{T-t}\, dM_t
\end{align}
an absolutely-continuous representation of $\xi$.
 {One of the weakest conditions} for the above to hold is  due to Veraar (see
Theorem 2.2 of \cite{Ver12}), and it can be stated in our case as
\begin{align}
   \label{equ:Veraar-cond}
\int_0^T  \prn{ \int_0^s \frac{1}{(T-t)^2} \,
   d \ab{M}_t}^{\tfrac{1}{2}}\, ds < \infty, \eas
\end{align}
It is superficially related to our condition 4.~of Theorem
\ref{thm:main-weak}, but it does not automatically insert $\beta$ into
our (weak or strong) regularity classes.
 {For example, assume that we are on a filtration generated
   by a Brownian
motion $W$, and that $\xi = W_T$. In that case} \eqref{equ:Veraar-cond}
is clearly satisfied, but, as we will see in Proposition
\ref{pro:markovian} below, $W_T$ does not admit an absolutely continuous
representation with a weakly (or strongly) regular $\beta$. Put
differently:
\begin{center}
\emph{regularity conditions for the validity of   the stochastic \\
Fubini theorem  do not correspond to our regularity classes.}
\end{center}
\end{remark}

\smallskip

A natural question is whether a condition such as:
\begin{enumerate}[label*=\arabic*., itemsep = 0.5em]
   \item[4'.] $\int_0^T \oo{T-t}\, d\ab{M}_t<\infty \eas$,
   for \underline{all} $M \in \sMl(\xi)$
\end{enumerate}
can be inserted in Theorem \ref{thm:main-weak}. An equivalent question
is whether the condition $4.$ of Theorem \ref{thm:main-weak} implies the
condition $4'.$ above. We only have a partial (positive) answer to this problem. It
states that under certain regularity conditions, if $M^\mathbb{P}$  satisfies condition 4 in Theorem \ref{thm:main-strong}, then all the  probability measures $\mathbb{Q} \sim \mathbb{P}$
with a finite relative entropy
share the property, namely $\eeqst{\int_0^T \oo{T-t}\, d\ab{M^{\bQ}}_t}<\infty$.

 {\begin{proposition}\label{pro:entropy}
Suppose that the filtration is generated by a
$\bP$-Brownian motion $W$ and that
$\xi$ is of the form
\[ \xi = \int_0^T \sigma_u\, dW_u,
    \text{ for some bounded $\sigma$.}
\]
If
\begin{enumerate}
    \item the martingale $M^{\bP}_t = \eep{ \xi  \giv \sF_t}$ satisfies
\[ \eep{\int_0^T \oo{T-t}\, d\ab*{M^{\bP}}_t}<\infty, \eand\]
\item $\bQ$ is a probability measure equivalent to $\bP$
    with a
    finite relative entropy, i.e.,
\begin{align}
   \label{equ:entr}
   \eep{ \RN{\bQ}{\bP} \log \prn{\RN{\bQ}{\bP}}}<\infty,
\end{align}
\end{enumerate}
then
\begin{align}
    \label{MQ-fin}
    \eeq{\int_0^T \oo{T-t}\, d\ab*{M^{\bQ}}_t}<\infty.
\end{align}
\end{proposition}
\begin{proof}
    With $\bP\sim\bQ$ as in the statement,
let $\theta$ be such that
the dynamics of  the density process $Z_t =
\eep{ \tRN{\bQ}{\bP} \giv \sF_t}$ is given by
\[ dZ_t = -Z_t\, \theta_t\, dW_t.\]
Let us show, first, that
\begin{align}
    \label{theta-l2}
\eeup{\bQ}{\int_0^T \theta^2_u\, du}<\infty.
\end{align}
Condition \eqref{equ:entr} above, together with the fact that the
function $x\mapsto x
\log(x)$ is convex and bounded from below, implies that the process
$Z\log(Z)$ is a
continuous $\bP$-submartingale on $[0,T]$. Hence, there exists a
finite constant
$C$ such that
\begin{align}
   \label{equ:tau-bound}
  \eep{ Z_{\tau} \log(Z_{\tau}) } \leq C
  \text{ for each $[0,T]$-valued stopping time $\tau$.}
\end{align}
\itos{} formula applied to the
semimartingale $Z \log(Z)$ yields
\begin{align*}
   Z_{\tau} \log(Z_{\tau}) = N_{\tau} + \tot \int_0^{\tau} Z_u
   \theta_u^2\, du \text{ where $N$ is a local $\bP$-martingale. }
   \end{align*}
Let  $\seq{\tau_n}$ is a sequence of stopping times that reduces the  process 
$N$. The upper bound of \eqref{equ:tau-bound} above
implies that
\begin{align*}
  C \geq \eep{Z_{\tau_n} \log(Z_{\tau_n})} = \tot \eep{\int_0^{\tau_n} Z_u \theta_u^2\, du}
  = \tot \eeq{ \int_0^{\tau_n} \theta_u^2\, du},
\end{align*}
and it remains to use the monotone-convergence theorem
to conclude that \eqref{theta-l2} holds.

% \my{Observe now that \eqref{equ:entr} can be equivalently formulated
% by $ \ee{ \RN{\bQ}{\bP} | \log \prn{\RN{\bQ}{\bP}}|}<\infty$, so that
% $E^\bQ[ \mid \int_0^T\theta _s dW_s + \frac{1}{2}\int_0^T \theta^2_s
% ds \mid] < \infty$. An application of the conditional Jensen's
% inequality then shows $\int_0^t\theta _s dW_s + \frac{1}{2}\int_0^t
% \theta^2_s ds $ is $\bQ$ integrable for all $t$.}

\smallskip

We continue the proof by using  Girsanov's theorem and the
boundedness of $\sigma$ to conclude that the process:
\[ M'_t =
\int_0^t \sigma_u\, (dW_u + \theta_u\, du) =
M^{\bP}_t + \int_0^t \sigma_u\theta_u\, du \]
is a $\bQ$-martingale.
Boundedness of the process $\sigma$, together with \eqref{theta-l2}, implies
that the random variable $\xi = \MP_T = M'_T - \int_0^T \sigma_u \theta_u\, du$ is
$\bQ$-integrable, so that the $\bQ$-martingale $\MQ_t = \eeq{\xi
\giv \sF_t}$ is well defined. Moreover, we have
\begin{align*}
M^{\bQ}_t & =
\eeq{ M'_T - \int_0^T \sigma_u \theta_u\, du \giv \sF_t}
= M^{\bP}_t + \int_0^t \sigma_u \theta_u\, du - L_t,
\end{align*}
where
\[ L_t = \eeq{\int_0^T \sigma_u \theta_u\, du \giv \sF_t}.\]
It follows that
\begin{align*}
   \ab*{M^{\bQ}}_t = \ab*{M^{\bP}-L}_t .
\end{align*}
Furthermore, since
\begin{align*}
    \ab*{M^{\bP}-L}_t  \leq \ab*{M^{\bP} - L}_t +
    \ab*{M^{\bP}+L}_t=    2\ab*{\MP}_t + 2\ab*{L}_t,
 \end{align*}
 our final goal, namely \eqref{MQ-fin}, will be reached if we can
 prove that
   \begin{align}
     \label{Z-int}
     \eeq{\int_0^T \oo{T-t}\, d\ab{L}_t}<\infty.
   \end{align}
   For that, we note that $\xi' = L_T$ admits the absolutely continuous
   representation
     $L_T = \textstyle\int_0^T \beta'_t\, dt$, where $\beta'_t = \sigma_t
     \theta_t$,
   by its very definition.
   Since the boundedness of $\sigma$ and \eqref{theta-l2} above imply that
   $\beta' \in \lto(\bQ)$, equation \eqref{Z-int}
follows from the implication $1.\to 4.$ of Theorem
\ref{thm:main-strong}.
    \end{proof}
}

\section{Examples and further remarks}
\label{sec:examples}
\subsection{Functions of the terminal value of a Brownian motion {W}}
  { Our first subsection focuses on the  (lack of) absolutely continuous
     representation property for the random variables of the form
     $g(W_T)$.
     In contrast with
     what happens with the martingale representation, we show that
     such a  $\xi$ admits an absolutely continuous representation if and
     only if it is constant.
    Informally speaking:  in order to have regular-enough
representation, $\xi$ must be sufficiently   path
dependent.

  We start with a simple argument of limited scope just to provide
  some intuition.   Let  $\xi$
 be the terminal value $W_T$ itself, and let $\beta$ be its
 absolutely continuous representation, i.e., $\xi = \int_0^T
 \beta_u\, du$.
% \begin{align*}
%    W_T = \int_0^T \beta_u\, du.
% \end{align*}
 Since the process $W^{\beta}_t = {W}_t - \int_0^t \beta_u\, du$
 satisfies $W^{\beta}_T = 0$, it cannot be a martingale
under any equivalent measure. This means that
$\beta$ cannot be too regular in the sense that
the stochastic
exponential $\exp*{ \int_0^{t} \beta_u\, dW_u - \tot \int_0^t \beta_u^2\,
du}$ cannot be a martingale,  if it is well-defined at all. In fact,
as our next Proposition shows, such a $\beta$ cannot be weakly
regular, i.e., $\int_0^T \beta_u^2\, du = +\infty$ with
positive probability. }

\begin{proposition}
   \label{pro:markovian}
    {Suppose that the filtration supports a Brownian motion
   $W$}, and that $\xi = g(W_T)$   for some $g\in C^2(\R)$.
Then $\sB^{2,0}(\xi) \not= \emptyset$ if and only if $g$ is constant.
\end{proposition}
\begin{proof}
The only implication which requires a proof
is $ (\Rightarrow)$.

 {We focus, first, on the case where the
filtration is generated by the Brownian motion.}
Since $g \in C^2(\mathbb{R})$, the process $
g''(W)$  is locally bounded so that $\int_0^T (g''(W_u))^2\, du<\infty$, a.s.  It\^ o's formula then
implies that $\xi$ admits an absolutely continuous representation under
weak regularity, i.e.,
that $\sB^{2,0}(\xi) \ne \emptyset$, if and only if $\sB^{2,0}(\bar{\xi})
\ne \emptyset$, where
  \begin{align*}
    \bar{\xi} = \int_0^T g'(W_t)\, dW_t.
  \end{align*}

   {Fix a $\bQ \in \sP$ and call $Z_t= \ee{ \tfrac{d\bQ}{d\bP}
      \giv \sF_t}$ its density process. Since
    the filtration is assumed to be Brownian, there exists a progressively
    measurable process $\theta$ such that
   $ Z=\mathcal{E}( \int_0^\cdot \theta_t dW_t)$. Thus,
  $$ \bar{\xi} = \int_0^T g'(W_t)\, dW_t = \int_0^T g'(W_t)\, dW^Z_t +  \int_0^T g'(W_t)\theta_t \, dt $$
where $W^Z = W - \int_0^\cdot \theta_t\, dt$ is a $\bQ$-Brownian
motion. By the continuity of $g'$ and the paths of $W$,
we have $\int_0^T (g'(W_t) \theta_t)^2\, dt <\infty$,
a.s.}
 %By  Theorem \ref{thm:main-weak},   when  $\bar \xi$ is representable  then there exists a $\bQ\sim P$ such that the $\bQ$ martingale which closes $\xi$,
% $$ \bar{\xi} = \eeq{\xi} +\int_0^T \sigma^{\bQ}(t) d W^{\bQ}(t)  $$
% has quadratic variation which satisfies $ \int_0^T\frac{ (\sigma^{\bQ}(t))^2}{T-t} dW^{\bQ}(t) <\infty  $.  By the Girsanov theorem however, there exists an  adapted process $\theta$ such that $W^{\bQ} = W -\int_0^\cdot \theta_t dt$.
Being deterministic, the quadratic variation of the Brownian motion
has the same distribution across all the probabilities in  $\sP$,
so, by  Theorem \ref{thm:main-weak},    $\bar \xi$ is representable if and only if:
\begin{align}
   \label{equ:g-prime-int}
  \int_0^T \frac{(g'(W_t))^2}{T-t}\, dt<\infty, \eas
\end{align}
Continuity of $g'$ and $W$ force $g'(W_t)\to 0$ when $t\rightarrow
T$, i.e., $g'(W_T)=0$, a.s., whenever
\eqref{equ:g-prime-int} holds. Since $W_T$ has full support under
any $\bP\in\sP$, we conclude that $g'\equiv 0$.\\
\indent  {The next step in the proof is to relax the assumption that the
filtration $\prf{\sF_t}$ is generated by a Brownian motion. In preparation, let
$\prf{\sF^W_t}$ denote the subfiltration of $\prf{\sF_t}$ generated by the Brownian motion
$W$, and let $\Prog$ and $\Prog^W$ denote the progressive
$\sigma$-algebras on $[0,T] \times \sF$ corresponding to
$\prf{\sF_t}$ and $\prf{\sF^W_t}$, respectively.

We argue by contradiction and assume that some $\xi = g(W_T)$, with
non constant $g$, can be
represented as $\int_0^T \beta_u\, du$ for some $\beta\in\Prog$ with $\int_0^T
\beta_u^2\, du < \infty$, a.s. With $\bQ$ denoting a probability
measure equivalent to $\bP$ with the property that $\eeqst{\int_0^T
\beta_u^2\, du}<\infty$, we observe that $\beta$, seen as a
$\Prog$-measurable function on the product space  $[0,T]\times
\Omega$, is square integrable with respect to the product probability
measure $\mu = \oo{T} \ld \otimes  {\bQ}$ (where $\ld$ denotes the Lebesgue
measure on $[0,T]$). Hence, the conditional expectation $\beta^W$ of
$\beta$, taken on
the probability space $([0,T]\times \Omega, \Prog, \mu)$, given the
$\sigma$-algebra $\Prog^W$, is well defined and satisfies
\begin{align*}
     \int_B \beta^W_u (\omega) \, d\mu(u, \omega) = \int_B \beta_u(\omega)\, d\mu(u, \omega) \eforall B\in \Prog^W.
\end{align*}
If we set $B = [0,T] \times A$, for $A\in \sF^W_T$ and remember that
$\xi \in \sF^W_T$, we obtain
immediately that $\int_0^T \beta^W_u\, du = \xi$, a.s.
Moreover, since the conditional expectation preserves square
integrability, we have
\begin{align*}
    \eeq{ \int_0^T (\beta^W_u)^2\, du} < \infty,
\end{align*}
and, consequently, $\int_0^T (\beta^W_u)^2\, du<\infty$, a.s. It
remains to observe that the existence of such $\beta^W$ contradicts
the conclusion of the first part of the proof, and completes the
argument.
}
\end{proof}

\subsection{Representations in $\bL^{p,1}$ for $p<2$}\label{sec: daprato-zab}
 {The ``factorization
formula'' of Da Prato and Zabczyk
when specialized to the case $S(t) = \Id$ and $U=H=\R$ (see Theorem 5.2.5,
p.~58 in \cite{DaPZab96} for the statement and the paragraph that precedes
it for the necessary definitions and notation) states the following:
whenever
$\alpha \in [0,1)$,
and $ {\Psi}$ is a progressively measurable process such that
\begin{align}
    \label{cond-fact}
    \int_0^t (t-s)^{\alpha-1} \prn{\int_0^s (s-r)^{-2\alpha} \ee{
    \Psi(r)^2}\, dr }^{1/2}\, ds < \infty \text{ a.s.,}
\end{align}
we have
\begin{align}
    \label{fact-formula}
    \int_0^t \Psi(s)\, dW_s = \frac{\sin(\alpha \pi)}{\pi} \int_0^t
    (t-s)^{\alpha-1} Y^{\Psi}_{\alpha}(s) \, ds \eas,  \eforall t\in [0,T],
\end{align}
where
\begin{align*}
    Y^{\Psi}_{\alpha}(s) = \int_0^t (s-r)^{-\alpha} \Psi(r)\, dW_r.
\end{align*}
To see how \eqref{fact-formula} above leads to an interesting absolutely
continuous representation, we pick a bounded progressively measurable process $\sigma$ and
apply the factorization formula \eqref{fact-formula} above with $t=T$ and $\Psi =
\sigma$. When $\alpha\in [0,1/2)$,  so that \eqref{cond-fact} is satisfied,
\eqref{fact-formula} yields directly the following
absolutely-continuous representation
for the last element $\xi=M_T$ of the martingale $M_t = \int_0^t
\sigma_u\, dW_u$:}
\begin{align}
   \label{equ:fact}
  \xi = \int_0^T \beta_t\, dt \ewhere
   {\beta_t =  \frac{\sin(\alpha \pi)}{\pi} (T-t)^{\alpha-1} R_t \eand R_t = \int_0^t
  (t-u)^{-\alpha}\, dM_u.}
\end{align}
Note that we can interpret  {$R_t$} as a
formal Riemann-Liouville fractional integral
of order $\alpha$ of the noise $dM$, up to a multiplicative constant.
To complete the factorization formula, we then integrate the result
multiplied by $(T-t)^{\alpha -1}$ with respect to $dt$, i.e., compute the
Riemann-Liouville integral of the complementary order $1-\alpha$ of the
result.
The semigroup property of Riemann-Liouville integration would
suggest that the result should coincide with the integral of order
$\alpha+(1-\alpha)=1$ of $dM$, i.e., it should yield $M_T=\xi$,   which  is precisely the case.
{We refer the reader to  \cite{SamKilMar93} for details
on fractional integration. }

Since no restriction other than boundedness is imposed on $\sigma$ it
may appear  {at first glance} that \eqref{equ:fact}
contradicts Proposition
\ref{pro:markovian} above,  {in that it provides a representation for $
\xi = W_T$, for example}. The difference, as in the discussion of the
stochastic Fubini theorem in Remark \ref{rem:Fubini} above, lies in the
regularity class.   {As an easy illustration, let us show that in the special case
$\sigma\equiv 1$, the representation $\beta$ given by \eqref{equ:fact}
does not even belong to the weak regularity class $\sB^{2,0}(\xi)$. In fact, we
show below that it satisfies $\int_0^T \beta_u^2\, du = +\infty$, a.s.
To do that, we focus on the process $R$ from \eqref{equ:fact}
which, in this case, takes the form
\begin{align}
    \label{RW}
 R_t = \int_0^t (t-s)^{-\alpha}\, dW_s.
\end{align}
Known as the \emph{Riemann-Liouville process} (up to a
multiplicative constant), the process $R$ from
\eqref{RW} above is
Gaussian and admits a
continuous modification (see \cite{MarRob99} for a survey
of its properties and its relation to the
fractional Brownian motion).
Its value $R_T$ at $T$ is normally
distributed with a nonzero variance, so that, by continuity,  we have $\lim_{t\to T}
R_t = R_T \ne 0$, a.s. It follows that there exists
a random variable $K$ such that $K>0$, a.s., and
\[ \abs{\beta_t} \geq K (T-t)^{\alpha-1}, \eas, \]
for all $t$ in a (random) neighborhood of $T$.
    Since $\alpha < 1/2$, i.e., $2\alpha -2 < -1$, we necessarily have $\int_0^T
    \beta_t^2\, dt = \infty$, a.s.
\\
\noindent On the other hand, as our next result states, the
factorization theorem allows us to construct absolutely continuous
representations of a wide variety of random variables
in any one of slightly less regular classes $\bL^{p,1}$,  $p \in [1,2)$.
}
 {
\begin{proposition}
    \label{lp}
Suppose that the filtration supports a Brownian motion $W$ and that
$\xi = \int_0^T \sigma_u\, dW_u$, with $\sigma$ bounded.
Then $\xi$ admits
an absolutely
continuous representation in $\lpo$ for each  $p\in [1,2)$.
\end{proposition}}
\begin{proof}
    Let $p\in [1,2)$ be given.
    We claim that that the representation \eqref{equ:fact} above, with $M_t =
    \int_0^t \sigma_u\, dW_u$, and with a
    suitably chosen $\alpha \in (0,1/2)$ belongs to $\lpo$. To see
    that we use Doob's maximal inequality followed by
the Burkholder-Davis-Gundy theorem to obtain the following:
\begin{align*}
  \ee{ \abs{\beta_t}^p } &= (T-t)^{p(\alpha-1)}
  \ee{ \abs{\int_0^t (t-u)^{-\alpha}\, dM_u}^p}\\
                         &  {\leq}  (T-t)^{p(\alpha-1)}
  \ee{ \sup_{s \leq t} \abs{\int_0^s (t-u)^{-\alpha}\, dM_u}^p}\\
  & \lessapprox (T-t)^{p(\alpha-1)} \ee{
     \prn{\int_0^t (t-u)^{-2  \alpha}\, d\ab{M}_u}^{p/2}}\\
\end{align*}
where $a \lessapprox b$ is a shorthand for   $a \leq C\,  b$, for some
constant $C>0$ which depends only on $p$.
Allowing $C$ to depend on  $\sigma$  as well, we can go on to
conclude that
\begin{equation}
   \label{equ:fin-inf}
   \begin{split}
   \ee{\int_0^T \abs{\beta_t}^p \, dt} & \leq
    {
      \int_0^T (T-t)^{p(\alpha-1)}  \ee{\prn{\int_0^t (t-u)^{-2\alpha }\, d\ab{M}_u}^{p/2}} dt
   }\\
   & \lessapprox
      \int_0^T (T-t)^{p(\alpha-1)}  \prn{\int_0^t (t-u)^{-2\alpha}\, du}^{p/2} \, dt
   \\
   & =
   \int_0^T (T-t)^{p(\alpha-1)}  \frac{t^{p(1/2-\alpha)}}{1-2\alpha}\, dt.
   \end{split}
\end{equation}
The last integral is finite if and only if
$\alpha \in (1-1/p,1/2)$. Therefore,
the choice $\alpha = \tfrac{3}{4} - \tfrac{1}{2p}$ in
\eqref{equ:fact} provides an absolutely continuous representation of
$\xi$ in $\bL^{p,1}$ for $p\in (1,2)$.
\end{proof}
\begin{remark}
\label{rem:Geoff}
The result of Proposition \ref{lp} above should be contrasted with  the findings in our Theorem \ref{thm:main-strong},
which implies that under the same assumptions, a representation of
$\xi = \int_0^T \sigma_u\, dW_u$ is possible in
$\lto$
if and only if, additionally,
$$\ee{\int_0^T \frac{\sigma_u^2}{T-t}\, dt}<\infty.$$
\end{remark}
\bibliographystyle{amsalpha} 
\bibliography{BiaZit23.bib}
\end{document}